\documentclass[a4paper]{amsart}

\usepackage{amssymb}
\usepackage{amsmath,amsthm}
\usepackage{enumerate,paralist}
\usepackage{amstext}
\usepackage{psfrag}
\usepackage{dsfont}
\usepackage[dvips]{epsfig}
\usepackage[english]{babel}

\theoremstyle{plain}
\newtheorem{theorem}{Theorem}
\newtheorem{corollary}{Corollary}
\newtheorem{lemma}{Lemma}
\theoremstyle{definition}
\newtheorem{definition}{Definition}
\newtheorem{example}{Example}
\newtheorem{remark}{Remark}

\numberwithin{theorem}{section}
\numberwithin{corollary}{section}
\numberwithin{lemma}{section}
\numberwithin{definition}{section}
\numberwithin{example}{section}
\numberwithin{remark}{section}
\renewcommand{\Re}{\ensuremath{{\rm Re\,}}}

\renewcommand{\leq}{\leqslant}
\renewcommand{\geq}{\geqslant}

\newcommand{\intl}{\int\limits}
\newcommand{\Tt}{(T_t)_{t\geq 0}}

\newcommand{\St}{(S_t)_{t\geq 0}}
\newcommand{\Xt}{(X_t)_{t\geq 0}}
\newcommand{\Yt}{(Y_t)_{t\geq 0}}
\newcommand{\Ee}{\mathds E}
\newcommand{\Pp}{\mathds P}

\newcommand{\nnorm}[1]{\|#1\|}
\newcommand{\1}{\mathds 1}
\newcommand{\comp}{\mathds C}

\newcommand{\real}{\mathds R}
\newcommand{\rn}{{{\mathds R}^d}}

\newcommand{\nat}{\mathds N}
\newcommand{\bu}{\cdot}
\newcommand{\cont}{C}
\newcommand{\ffi}{\varphi}
\newcommand{\eps}{\varepsilon}

\newcommand{\supp}{\mathop{\mathrm{supp}}\nolimits}
\newcommand{\id}{\mathop{\mathrm{Id}}\nolimits}

\begin{document}
\frenchspacing

\title[Lagrangian and Hamiltonian Feynman formulae]{Lagrangian and Hamiltonian Feynman formulae for some Feller semigroups and their perturbations}

\author{YANA A.\ BUTKO}

\address{Department of Fundamental Sciences, Bauman Moscow State Technical
University\\105005, 2nd Baumanskaya str., 5, Moscow, Russia\\
yanabutko@yandex.ru}

\author{REN\'E L.\ SCHILLING}

\address{Institut f\"ur Mathematische Stochastik, Technische Universit\"at Dresden\\
Zellescher Weg, 12-14, D-01069 Dresden,
Germany\\
rene.schilling@tu-dresden.de}

\author{OLEG G.\ SMOLYANOV}

\address{Department of Mechanics and Mathematics,
Lomonosov Moscow State University\\119992, Vorob'evy gory, 1, Moscow, Russia\\
Smolyanov@yandex.ru}

\maketitle

\begin{abstract}
A Feynman formula is a representation of a solution of an initial
(or initial-boundary) value problem for an evolution equation (or,
equivalently, a representation of the semigroup resolving the
problem) by a limit of $n$-fold iterated integrals of some
elementary functions as $n\to\infty$. In this note we obtain some
Feynman formulae for a class of semigroups  associated with Feller
processes.
 Finite dimensional integrals in the Feynman formulae
give  approximations for functional integrals in some Feynman--Kac
formulae corresponding to the underlying processes. Hence, these Feynman
formulae give an effective  tool to calculate functional integrals
 with
respect to probability measures generated by these Feller processes
and, in particular, to obtain  simulations of Feller
processes.


\end{abstract}


\section{\bfseries Introduction}
In this note we consider a class of semigroups  associated with Feller
processes. Feller  processes are
continuous-time Markov processes, which generalize the class of
stochastic processes with stationary and independent increments, or
L\'{e}vy processes. Note that many diffusion processes  belong
to this class.  Every Feller process $(\xi_t)_{t\ge0}$ in $\real^d$
generates a strongly continuous positivity preserving contraction
semigroup $(T_t)_{t\ge0}$  on  the space $C_\infty(\real^d)$ of
continuous functions vanishing at infinity: $T_tf(q)=
\mathds{E}^q[f(\xi_t)]$ for any $f\in C_\infty(\real^d)$. Due to
Courr\`ege it is  known that (under a mild richness condition on
the domain) the infinitesimal generator $A$ of a Feller semigroup is
a pseudo-differential operator ($\Psi$DO, for short), i.e.\ an
operator of the form
$$
    Af(q)= -H(q,D)f(q) = -(2\pi)^{-d}\intl_{\real^d}\intl_{\real^d}e^{i(q-x)\cdot p} H(q,p)f(x)\,dx\,dp,
    \quad f\in C_c^\infty(\real^d).
$$
The symbol of the operator, the function
$-H:\real^d\times\real^d\to\mathds{C}$, $(q,p)\mapsto -H(q,p)$, is for
fixed $q$, given in terms of a L\'evy-Khintchine representation
\begin{gather*}
   H(q,p)
    = a(q) + i\ell(q)\cdot p +  p\cdot Q(q)p + \intl_{y\neq 0} \left(1-e^{i p\cdot y} +
     \frac{i p\cdot
     y}{1+|y|^2}\right)\,N(q,dy),
\end{gather*}
where, for each fixed $q$,  $\ell(q)\in\real^d$, $Q(q)$ is a positive
semidefinite symmetric matrix    and $N(q,dy)$ is a measure kernel on
 $\real^d\setminus\{0\}$ such that $\int_{y\neq 0} \frac{|y|^2}{1+|y|^2}\,N(q,dy)<\infty$. Note that these
 \emph{negative definite symbols} do not belong to any of the
  classical symbol classes of $\Psi$DOs; consequently we do not have a H\"{o}rmander or Maslov symbolic calculus at our disposal.

   In a similar way, each operator $T_t$ can be represented as a pseudo-differential operator $\lambda_t(\cdot,D)$ with the symbol
    $\lambda_t(q,p) = \mathds{E}^q\big[ e^{i(\xi_t-q)\cdot p}\big]$. It is known that $-H(q,p)=\lim_{t\to0} \frac{\lambda_t(q,p)-1}{t}$,
     see e.g.\ \cite{jac-pota,schilling-pota,schilling-schnurr}. If $(\xi_t)_{t\ge 0}$ is a L\'{e}vy process, we have $H(q,p)=H(p)$ and
      $\lambda_t(q,p)=e^{-tH(p)}$ --- this is due to the fact that the generator is an operator with constant ``coefficients''(i.e. it is independent
      of the state space variable $q$). In the  case where we have variable ``coefficients'', there is no such straightforward connection
      between the symbols of the semigroup and the generator and this gives rise to several interesting problems:
       which negative definite symbols $H(q,p)$ lead to Feller processes and, if so, how can we represent
        or approximate the symbol $\lambda_t(q,p)$. The existence problem has been discussed at length in a series
        of papers, see \cite{NJ,schilling-boettcher,JSch} and the literature given there,
        and we would like now to investigate the problem how to represent the semigroup resp.\ its symbol
        if the (symbol of the) generator is known.

Consider an evolution equation $\frac{\partial f}{\partial
t}(t,q)=-H(q,D)f(t,q)$, where  $-H(\cdot,D)$ is a generator of some
Feller process $(\xi_t)_{t\ge0}$. Following the terminology of
mathematical physics, we call $H(\cdot,D)$ the \emph{Hamiltonian} of
the physical system, which is described by the above evolution
equation. The solution of the Cauchy problem for this equation with
 initial data $f(0,q)=f_0(q)$ can be obtained by the Feynman-Kac
formula $f(t,q)\equiv (T_tf_0)(q)=\mathds{E}^q[f_0(\xi_t)]$. Here
the expectation $\mathds{E}^q[f_0(\xi_t)]$ is a functional integral (path
integral) over the set of paths of the process $(\xi_t)_{t\ge0}$ with
respect to the measure generated by this process. If
$(\xi_t)_{t\ge0}$ is a diffusion process, then
$\mathds{E}^q[f_0(\xi_t)]=\int_{C([0,t],
\real^d)}f_0(\xi_t)\mu(d\xi)$, where $\mu$ is a Gaussian measure,
corresponding to this process; in particular, the Wiener measure
corresponds to the process of Brownian motion.

The heuristic notion of a path integral has been introduced by R.
Feynman, see \cite{Feyn1}, \cite{Feyn2}, to obtain the
solution of the Schr\"{o}dinger equation with a potential. Feynman
has defined this integral as a limit of some finite dimensional
integrals; actually, these integrals range over Cartesian powers of
the configuration (or phase) space of the system, described by the
Schr\"{o}dinger equation. In  modern terminology this kind of
path integrals are called \emph{Feynman path integrals with respect
to a Feynman pseudomeasure on the set of paths in the configuration
(or phase) space}.

The classical Feynman-Kac formula, representing the solution of the
Cauchy problem for the heat equation by a functional integral with
respect to the Wiener measure can also be obtained applying
Feynman's construction. Here the functional integral is a limit of
$n$-fold iterated integrals containing Gaussian exponents which are
transition densities of a  Brownian motion.
This construction can be extended to a  large class of Markov processes.
However, in
most cases the transition densities of Feller processes cannot be
expressed by elementary functions and, hence, in order to compute
functional integrals in Feynman-Kac formulae  we need to approximate them. This gives rise
to Feynman formulae.

A \emph{Feynman formula} is a representation of the solution of an
initial (or initial--boundary) value problem for an evolution
equation (or, equivalently, a representation of the semigroup
resolving the problem) as a limit of $n$-fold iterated integrals of
some elementary functions, when $n\to\infty$. Obviously, the
iterated integrals in a Feynman formula for some problem give
approximations for a functional integral in the Feynman-Kac formula
representing the solution of the  problem. These approximations
can be used for direct calculations and simulations.

The notion of  a Feynman formula has been introduced in
\cite{STT} and the method to obtain Feynman formulae for
evolutionary equations has been developed in a series of papers
\cite{STT}--\cite{SWW5}.  Recently, this method has been
successfully applied to obtain Feynman formulae for different
classes of problems for evolutionary equations on different
geometric structures, see, e.g. \cite{MZ2}--\cite{BGS1},
\cite{Obr1}, \cite{OST}, \cite{SSham} and also to construct some surface
measures on infinite dimensional manifolds (see \cite{SWW1}--\cite{Telyat1}). This
method is based on Chernoff's theorem (see \cite{Ch2} and
\cite{STT} for the version used here), which is a
generalization of the well-known Trotter formula. Trotter's formula
has been used to justify Feynman's heuristic result for
Schr\"{o}dinger equations with a potential, e.g.\
\cite{Nelson}, and to prove the classical Feynman-Kac formula
mentioned earlier.

By Chernoff's theorem a strongly continuous semigroup
$(T_t)_{t\ge0}$ on a Banach space can be represented as a strong
limit: $T_t=\lim_{n\to\infty}[F(t/n)]^n$ where $F(t)$ is an
operator-valued function satisfying certain conditions (see Theorem
\ref{Chernoff} for details). This equality is called a \emph{Feynman
formula} for the semigroup $(T_t)_{t\ge0}$. We call this  Feynman
formula a \emph{Lagrangian Feynman formula}, if the $F(t)$, $t>0$,
are integral operators with elementary kernels; if the $F(t)$ are
$\Psi$DOs, we speak of \emph{Hamiltonian Feynman formulae}. In
particular, we obtain a Hamiltonian Feynman formula for a semigroup
$T_t\equiv e^{-tH(\cdot,D)}$ generated by a $\Psi$DO $-H(\cdot,D)$
with the symbol $-H(q,p)$ if
$$
    e^{-tH(\cdot,D)}= \lim_{n\to\infty}\big[e^{-\frac{t}{n}H}(\cdot,D)\big]^n,
$$
where $e^{-\frac{t}{n}H}(\cdot,D)$ is the $\Psi$DO with the symbol
${e^{-\frac{t}{n}H(q,p)}}$. Note that, in general
$e^{-\frac{t}{n}H}(\cdot,D)$ is not a semigroup and that
$\lambda_t(q,p)\ne e^{-tH(q,p)}$.

Our terminology is inspired by the fact that a Lagrangian Feynman
formula gives approximations to a functional integral over a set of
paths in the configuration space of a system (whose evolution is
described by the semigroup $(T_t)_{t \ge 0}$), while a Hamiltonian
Feynman formula corresponds to a functional (Hamiltonian Feynman path) integral over a set of
paths in the phase space of  some system (cf. \cite{BBSchS}).  The corresponding Hamiltonian Feynman formula
gives rise to a Hamiltonian Feynman path integral also in the case
$T_t=e^{it{H(q,D)}}$ (see \cite{STT}).

In this note we prove some Hamiltonian and Lagrangian Feynman
formulae for  semigroups associated with Feller processes and for perturbations of such semigroups. Several
results of the paper have been announced in \cite{BShS}. The
paper is organized as follows. Section 2 contains notation
and some preliminaries; in particular, Chernoff's theorem is
formulated and the notion of Chernoff equivalence is introduced. In
Section 3 we prove a Hamiltonian Feynman formula for a  class
of semigroups associated with Feller processes.
In Section 4 we obtain a Lagrangian Feynman formula  for a multiplicative perturbation of a Feller
semigroup by  a  function $a(\cdot)$ which is continuous, positive, bounded and bounded away from zero.
 Note, that analogous Lagrangian Feynman formulas  have been proved  for some diffusion
 processes in \cite{BGS2} and have been presented for  the Cauchy process in \cite{BShS}.  In
 Section 5 we consider gradient and bounded  Schr\"{o}\-din\-ger
 perturbations of Feller semigroups and obtain some Hamiltonian and
 Lagrangian Feynman formulae for them.

\section{\bfseries Notations and preliminaries}
 Let $\cont^\infty_c(\mathds{R}^d)$ be a set of infinitely
differentiable  functions on $\mathds{R}^d$ with compact support and
$S(\real^d)$ be the Schwartz space of rapidly decreasing functions.
Let us also consider a space $\cont_\infty(\mathds{R}^d)$ of all
continuous functions vanishing at infinity. It is a Banach space
with the norm $\|f\|_\infty=\sup_{x\in \mathds{R}^d}|f(x)|$. Write
for the norm in $C^k_\infty(\rn)$, the space of $k$ times
continuously differentiable functions which vanish (with all their
derivatives) at infinity,
$$
    \nnorm u_{(k)} := \sum_{|\alpha| \leq k}\nnorm{\partial^\alpha u}_\infty
$$
where $\alpha \in \nat_0^n$, $\partial^\alpha =
\partial^{|\alpha|}/\partial x_1^{\alpha_1}\cdots\partial
x_n^{\alpha_n}$, and $|\alpha| = \alpha_1 +\ldots+\alpha_n$.

We use the following notations for the   Fourier transform and its inverse:
$$\widehat{f}(p)=({2\pi})^{-d/2}\intl_{\real^d}e^{-ip\cdot q}f(q)dq \quad\text{and}\quad
\mathcal{F}^{-1}[f](q)=({2\pi})^{-d/2}\intl_{\real^d}e^{ip\cdot q}f(p)dp.$$

\subsection{Negative definite functions.}

 Negative definite functions have been introduced by I.J.\ Sch\"onberg in
connection with isometric embeddings of metric spaces into a Hilbert
space. His original definition is the following.
\begin{definition}\label{d2}
    A function $\psi : \rn \to \comp$ is called \emph{negative definite}
    if for any $m\in\nat$ and all $p_1, \ldots, p_m \in\rn$ the $m\times m$ matrix
$
    \left( \psi(p_j) + \overline{\psi(p_k)} - \psi(p_j-p_k)
    \right)_{j,k=1,\ldots, m}
$ is positive hermitian, i.e., if for all $\lambda_1, \ldots,
\lambda_m \in \comp$
$$
    \sum_{j,k=1}^m \left(\psi(p_j) + \overline{\psi(p_k)} - \psi(p_j-p_k)\right) \,\lambda_j \overline{\lambda_k} \geq 0 .
$$
\end{definition}
 A \emph{negative} definite function is NOT the negative of a
\emph{positive definite} function.  Recall that a function $u:\real^d\to\comp$ is called \emph{positive definite} if
for any choice of $k\in\mathds{N}$ and vectors $p_1,\ldots ,p_k\in\real^d$ the matrix $(u(p_i-p_j))_{i,j=1,\ldots ,k}$ is positive Hermitian, i.e. for all $\lambda_1,\ldots ,\lambda_k\in\comp$ we have $\sum_{i,j=1}^k u(p_i-p_j)\lambda_i\overline{\lambda}_j\ge0$.

\begin{corollary}\label{cor}
If $u:\real^d\to\comp$ is a positive definite function, then the
function $[p\mapsto u(0)-u(p)]$ is negative definite.
\end{corollary}

The deeper connection between positive definite and negative
definite functions can be seen from the following Theorem \ref{s3}
which also justifies the definition of continuous negative definite
functions through the L\'evy-Khintchine formula:
\begin{definition}
A function  $\psi : \rn \to \comp$ is called  a \emph{continuous
negative definite function} if $\psi$ is given by the
\emph{L\'evy-Khintchine formula}
\begin{equation}\label{g16}
    \psi(p)
    = a +i\ell\cdot p + p\cdot Q p
      + \intl_{y\neq 0}
        \left(1-e^{iy\cdot p} + \frac{iy\cdot p}{1+|y|^2}
        \right)\,N(dy).
\end{equation}
 The tuple $(a,\ell,Q,N)$ consisting of $a\in \mathds{R}^+$,
$\ell\in\rn$, a positive semidefinite matrix
$Q\in\mathds{R}^{d\times d}$ and a Radon measure $N$ on
$\rn\setminus\{0\}$ with $\int_{y\neq 0}
|y|^2(1+|y|^2)^{-1}\,N(dy)<\infty$ is called \emph{L\'evy
characteristics (of $\psi$)}. The measure $N$ is often called
\emph{L\'evy measure}.
\end{definition}
Obviously, the L\'evy characteristics are uniquely determined by  $\psi$---and vice versa.

\begin{theorem}\label{s3}
    For $\psi:\rn\to\comp$ the following properties are equivalent:
    \begin{compactenum}[\upshape\indent (a)]
    \item
     $\psi$ is continuous and negative definite in the sense of
                    Definition \ref{d2};
    \item $\psi$ is given by the L\'evy-Khintchine formula \emph{(\ref{g16})};
    \item $\displaystyle\psi(0)\geq 0$ and $e^{-t\psi}$ is for every $t>0$
                    continuous and positive definite.
    \end{compactenum}
\end{theorem}
 A proof of Theorem \ref{s3} can be found, e.g. in the monographs by Jacob \cite{NJ} or by
Berg and Forst \cite{ber-for}(II.\S7).  All continuous positive definite functions are characterized by Bochner's Theorem.
\begin{theorem}[Bochner]\label{s4}
    A function $\phi:\rn\to\comp$ is continuous and positive definite if, and only if,
    it is the Fourier transform of a bounded Radon measure $\mu\in\mathcal M_b^+(\rn)$, i.e.,
$$
    \phi(p) = \widehat\mu(p) := (2\pi)^{-d/2}\int_{\real^d} e^{-ip\cdot q}\,\mu(dq).
$$
\end{theorem}

From Definition \ref{d2} it is not hard to see that a negative
definite function has positive real part $\Re\psi\geq 0$, satisfies
$\overline{\psi(p)} = \psi(-p)$ and that $\sqrt{|\psi(\bu)|}$ is
subadditive, i.e.,
$$
    \sqrt{|\psi(p_1+p_2)|}
    \leq \sqrt{|\psi(p_1)|} + \sqrt{|\psi(p_2)|}, \qquad p_1, p_2\in\rn.
$$
If $\psi$ is continuous, repeated applications of the subadditivity
estimate yield the following growth bound
\begin{equation}\label{g25}
    |\psi(p)| \leq 2\,\sup_{|\eta|\leq 1}|\psi(\eta)|\,\left(1+|p|^2\right),
    \qquad p\in\rn.
\end{equation}

\subsection{Feller and L\'{e}vy semigroups and their generators}
A \emph{Feller process} $\Xt$ with a state space $\rn$ is a strong
Markov process whose associated operator semigroup $\Tt$,
$$
    T_tu(x) = \Ee^x\left[u(X_t)\right],
    \qquad u\in C_\infty(\rn), \; t\geq 0,\; x\in\rn,
$$
 enjoys the \emph{Feller property}, i.e., it is a strongly continuous positivity preserving contraction
semigroup 
on  the space $C_\infty(\real^d)$.  The semigroup $\Tt$ is said to be a
\emph{Feller semigroup}.


The (\emph{infinitesimal}) \emph{generator} $(A,D(A))$ of
the semigroup or the process is given by the strong limit
$$
    Au := \lim_{t\to 0}\frac{T_t u - u}{t}
$$
on the set $D(A)\subset C_\infty(\rn)$ of those $u\in
C_\infty(\rn)$ for which the above limit exists w.r.t.\ the
$\sup$-norm. We will call $(A,D(A))$ a \emph{Feller
generator} for short.

Before we proceed with general Feller semigroups it is instructive
to have a brief look at L\'evy processes (and
convolution semigroups) which are a particular subclass of
Feller processes. Our standard reference for L\'evy processes is the
monograph by K.\ Sato \cite{Sato}. A \emph{L\'evy process} $\Yt$ is a
stochastically continuous random process with stationary and
independent increments. The Fourier transform of a L\'evy process
has a particularly simple structure,
\begin{equation}\label{g15}
    \Ee^x\left[ e^{ip\cdot(Y_t - x)}\right]
    = \Ee^0\left[ e^{ip\cdot Y_t}\right]
    = e^{-t\psi(p)},
\end{equation}
where $\psi : \rn \to \comp$ is the \emph{characteristic exponent}
which is a \emph{continuous negative definite function}, i.e. $\psi$
is given by the \emph{L\'evy-Khintchine formula} \eqref{g16}.
 Since $\Yt$ is a Markov process both
(\ref{g15}) and (\ref{g16}) characterize the finite dimensional
distributions of $\Yt$ and, hence, the process itself.

A L\'evy process is spatially homogeneous. Therefore, the associated
semigroup is of convolution type,
$$
    S_t u(x)
    = \Ee^x\left[u(Y_t)\right]
    = \Ee^0\left[u(Y_t+x)\right]
    = \intl_{\mathbb{R}^d} u(x+y)\,\Pp^0(Y_t\in dy)
    = u* \tilde\mu_t(dy),
$$
$\tilde\mu_t(dy) = \Pp^0(Y_t\in -dy)$, and a short direct
calculation shows that $\St$ is indeed a Feller semigroup with
infinitesimal generator
\begin{equation}\label{g110}
    Bu(x) = -\psi(D)u(x)
    := -(2\pi)^{-n/2} \intl_{\mathbb{R}^d} \psi(p) \, \widehat u(p)\,e^{ix\cdot p}\,dp,
    \qquad u\in C_c^\infty(\rn).
\end{equation}
One can use the estimate (\ref{g25}) to show that integrals in
(\ref{g110}) are  convergent.

 The operator $\psi(D)$ is a
first example of a so-called \emph{pseudo differential operator}
with  the \emph{symbol} $ \psi(p)$. Since $\psi$ does not depend on
$x$ the operator has constant ``coefficients''. Notice that the
symbol $\psi$ is just the characteristic exponent of the process
$\Yt$. This shows that

\bigskip\noindent
\emph{every L\'evy process is generated by a pseudo differential
operator $-\psi(D)$ with the  symbol $-\psi(p)$  where  $\psi $ is
the characteristic exponent of the process. Conversely, every pseudo
differential operator $-\psi(D)$ with the symbol $-\psi(p)$,  where
$\psi$ is a continuous negative definite  function, i.e., given by
the L\'evy-Khintchine formula {\upshape (\ref{g16})}, is the
generator of a L\'evy process.}

\bigskip
Let us return to the general situation. It is not hard to see
(cf.~\cite{eth-kur},p.\ 165, Theorem 2.2\,(b)) that Feller generators
satisfy the so-called \emph{positive maximum principle}
$$
    \textrm{if}\qquad
    u\in D(A),\quad \sup_{x\in\rn} u(x) = u(x_0)\geq 0
    \qquad\textrm{then}\qquad
    Au(x_0)\leq 0.
\leqno{\textrm{(PMP)}}$$ Extending earlier work of W.\ von
Waldenfels ~\cite{wald61,wald64} Ph.\ Courr\`ege showed in \cite{cour66}, see also \cite{bon-cour-pri68}, the
following structure result for operators satisfying the positive
maximum principle. We formulate his theorem only for Feller
generators.
\begin{theorem}[Courr\`ege]\label{s1}
    Let $(A,{D}(A))$ be a Feller generator
     such that $C_c^\infty(\rn)\subset {D}(A)$.
     Then $A\big|_{C_c^\infty(\rn)}$ is a pseudo differential operator,
\begin{equation}\label{g115}
    Au(q)
    = -H(q,D)u(q)
    =- (2\pi)^{-n/2}\intl_{\mathbb{R}^d} H(q,p)\, \widehat u(p)\,e^{ip\cdot q}\,dp,
    \qquad u\in C_c^\infty(\rn),
\end{equation}
with the symbol $H:\rn\times\rn\to\comp$ which is measurable,
locally bounded in both variables $(q,p)$, and satisfies for fixed
$q$ the following L\'evy-Khintchine representation
\begin{equation}\label{g116}
    H(q,p)
    = a(q) +i\ell(q)\cdot p + p\cdot Q(q)p
      + \intl_{y\neq 0}
        \left(1-e^{iy\cdot p} + \frac{iy\cdot p}{1+|y|^2} \right)\,N(q,dy),
\end{equation}
where $(a(q),\ell(q),Q(q),N(q,\bu))$ is for each $q\in\rn$ the
L\'evy characteristics of $-H(q,\bu)$. \end{theorem}

Observe that (\ref{g116}) automatically implies the continuity of
$p\mapsto H(q,p)$ for each $q\in\rn$.

Let $H(q,D)$ be a pseudo differential operator with the symbol $H(q,
p)$ as in Theorem \ref{s1}. Since $H(q, p)$ is represented by the
L\'evy-Khintchine type formula (\ref{g116}) we can use Fourier
inversion in (\ref{g115}) and find that the integro-differential
operator
\begin{equation}\label{g35}
\begin{split}
    A\ffi(q)
    &=
    -a(q)\ffi(q) + \ell(q)\cdot\nabla \ffi(q)
    + \sum_{j,k=1}^d Q^{jk}(x)\partial_j\partial_k \ffi(q)\\
    &\phantom{==}+ \int_{y\neq 0} \left( \ffi(q+y) - \ffi(q)
    - \frac{y\cdot\nabla \ffi(q)}{1+|y|^2}\right)\,N(q,dy)
\end{split}
\end{equation}
extends $\left(-H(\bu,D),C_c^\infty(\rn)\right)$ to the set
$C^2_\infty(\real^n)$. Note that the following Lemma \ref{l2} together
with the integration properties of $N(q,dy)$,
$$\intl_{y\neq 0}
|y|^2/(1+|y|^2)\,N(q,dy) < \infty,$$
 ensure that the integral in
(\ref{g35}) converges. From now on we will use the pseudo
differential representation (\ref{g115}) and the
integro-differential representation (\ref{g35}) simultaneously.

\begin{lemma}\label{l2}
    For all $\ffi\in C_b^2(\rn)$ we have
\begin{equation}\label{g310}
    \left|\ffi(q+y)-\ffi(q)-\frac{y\cdot\nabla \ffi(q)}{1+|y|^2} \right|
    \leq 2\,\frac{|y|^2}{1+|y|^2}\,\nnorm{\ffi}_{(2)}.
\end{equation}
\end{lemma}
\begin{proof}
    By Taylor's formula we get for all $q,y\in\rn$
\begin{align*}
    \bigg| (1+|y|^2)&\left( \ffi(q+y)-\ffi(q)-\frac{y\cdot\nabla \ffi(q)}{1+|y|^2}\right)\bigg|\\
    &\leq \left|\ffi(q+y)-\ffi(q)-y\cdot\nabla \ffi(q)\right| + |y|^2\left| \ffi(q+y)-\ffi(q)\right|\\
    &\leq \frac 12  \left|\sum_{j,k=1}^d y_jy_k\partial_j\partial_k \ffi(\xi_{q,y})\right|
      + 2|y|^2\nnorm \ffi_\infty\\
    &\leq 2|y|^2\left(\nnorm \ffi_\infty +\sqrt{ \sum_{j,k=1}^d \nnorm{\partial_j\partial_k \ffi}_\infty^2}
          \right)\\
    &\leq 2|y|^2\,\nnorm \ffi_{(2)}.
\qedhere\end{align*}
\end{proof} 


In the sequel we will need also the following Lemma.
\begin{lemma}\label{l1}
    We have
$$
    \frac{|y|^2}{1+|y|^2} = \intl_{\rn} \left(1-\cos(y\cdot p)\right)\, g(p)\,dp,
    \qquad y\in\rn,
$$
    where
$
    g(p) = \frac 12 \int_0^\infty
            (2\pi\lambda)^{-d/2}\,e^{-|p|^2/2\lambda}\,e^{-\lambda/2}\,d\lambda
$
    is integrable and has absolute moments of arbitrary order.
\end{lemma}
\begin{proof}
   The Tonelli-Fubini Theorem and a change of variables
show for $k\in\nat_0$
\begin{align*}
    \intl_{\real^d} |p|^k\,g(p)\,dp
    &=
    \frac{1}{2} \intl_0^\infty (2\pi\lambda)^{-d/2} \intl_{\rn} |p|^k\,e^{-|p|^2/2\lambda}
        \,dp \; e^{-\lambda/2}\,d\lambda\\
    &=
    \frac{1}{2} \intl_0^\infty (2\pi\lambda)^{-d/2} \intl_{\rn} \lambda^{k/2} |\eta|^k
        \,e^{-|\eta|^2/2} \lambda^{d/2}\,d\eta \; e^{-\lambda/2}\,d\lambda\\
    &=
    \frac{1}{2} (2\pi)^{-d/2} \intl_{\rn} |\eta|^k \,e^{-|\eta|^2/2}\,d\eta\
    \intl_0^\infty \lambda^{k/2}\,e^{-\lambda/2}\,d\lambda ,
\end{align*}
i.e., $g$ has absolute moments of any order. Moreover, the
elementary formula
$$
    e^{-\lambda |y|^2/2} = (2\pi\lambda)^{-d/2}\intl_{\real^n}
                            e^{-|p|^2/2\lambda} \, e^{iy\cdot p}\, dp
$$
and Fubini's Theorem yield
\begin{align*}
    \frac{|y|^2}{1+|y|^2}
&=
    \frac{1}{2}\intl_0^\infty \big( 1-e^{-\lambda |y|^2/2}\big)\,e^{-\lambda/2}\,d\lambda\\   &=
    \frac{1}{2}\intl_0^\infty\intl_{\real^d} (2\pi\lambda)^{-n/2} \big(1-e^{iy\cdot p}\big)
    e^{-|p|^2/2\lambda}\, e^{-\lambda/2}\,dp\,d\lambda\\
&=
    \intl_0^\infty \big(1-e^{iy\cdot p}\big)\,g(p)\,dp .
\end{align*}
The assertion follows since the left-hand side is real-valued.
\end{proof}

\subsection{The Chernoff theorem}

If $X,X_1,X_2$ are Banach spaces, then $L(X_1,X_2)$ denotes the
space of continuous linear mappings from $X_1$ to $X_2$ equipped
with the strong operator topology, $L(X)=L(X,X)$, $\|\cdot\|$
denotes the operator norm on $L(X)$ and $\id$ the identity operator
in $X$. If $D(T) \subset X$ is a linear subspace and $T: D(T) \to X$
is a linear operator, then $D(T)$ denotes the domain of $T$.

The derivative at the origin of a function $F:[0,\varepsilon) \to
L(X)$, $\varepsilon > 0$, is a linear mapping $F'(0): D(F'(0)) \to
X$ such that
\begin{eqnarray*}
F'(0)g := \lim_{t \searrow 0}\frac{F(t)g -F(0)g}{t},
\end{eqnarray*}
where $D(F'(0))$ is the vector space of all elements $g \in X$ for
which the above limit exists.

In the sequel we use the following version of Chernoff's theorem
(see \cite{STT}).
\begin{theorem}[Chernoff]\label{Chernoff}
Let $X$ be a Banach space, $F:[0,\infty)\to{L}(X)$ be a (strongly)
continuous mapping such that $F(0) =\id$  and $\|F(t)\|\le e^{at}$
for some  $ a\in [0, \infty)$ and all $t \ge 0$. Let
 $D$ be a linear subspace of  $D(F'(0))$ such that the restriction of the operator
 $F'(0)$ to this subspace is closable. Let $(A, D(A))$ be this closure. If
 $(A, D(A))$ is the generator of a strongly continuous semigroup
 $(T_t)_{t \ge 0}$, then for any $t_0 >0$ the sequence
 $(F(t/n))^n)_{n \in {\mathds N}}$ converges to $(T_t)_{t \ge 0}$ as $n\to\infty$
 in the strong operator topology, uniformly with respect to
 $t\in[0,t_0],$ i.e., $T_t = \lim_{n\to\infty}(F(t/n))^n$.
\end{theorem}
A family of operators $(F(t))_{t \ge 0}$ is called \emph{Chernoff
equivalent} to the semigroup $(T_t)_{t \ge 0}$ if this family
satisfies the assertions of Chernoff's theorem;
then, by Chernoff's theorem we have in $L(X)$  locally uniformly with respect to $t$
\begin{equation}\label{FF}
T_t = \lim_{n\to\infty}(F(t/n))^n. \end{equation}
 The  equality \eqref{FF} is called
\emph{Feynman formula} for the semigroup $(T_t)_{t \ge 0}$.

\section{\bfseries Hamiltonian Feynman formula for some Feller semigroups}


Consider a function  $H:\rn\times\rn\to\comp$ which is  measurable, locally bounded in both
variables $(q,p)$, and satisfies for fixed $q$ the L\'evy-Khintchine
representation \eqref{g116}, i.e. $H(q,\bu)$ is a continuous
negative definite function for all $q\in\rn$. Assume that
\begin{equation}\label{a}
\displaystyle\sup_{q\in\rn} |H(q,p)| \leq
                 \kappa(1+|p|^2)\quad \text{for all}\quad p\in\rn \quad\text{and some}\quad \kappa>0 ,
\end{equation}
\begin{equation}\label{b}
  \displaystyle p\mapsto H(q,p)\quad \text{is uniformly (w.r.t.} \quad q\in\rn \text{)
                 continuous at}\quad p=0,
\end{equation}
\begin{equation}\label{c}
\displaystyle q\mapsto H(q,p)\quad \text{is continuous for
all}\quad p\in\rn.
\end{equation}


 Consider a  $\Psi$DO
${H}(\cdot,D)$ with the symbol $H(q,p)$, i.e. for each $\ffi\in
C_c^\infty(\real^d)$
we have
\begin{equation}\label{H}
{H}(q,D)\ffi(q)=(2\pi)^{-d/2}\intl_{\real^d}e^{ip\cdot q}H(q,p)\widehat{\ffi}(p)dp.
\end{equation}
Note that  (due to the estimate \eqref{g25}) the condition \eqref{a} actually means that a $\Psi$DO
$H(\cdot,D)$ is an operator with  bounded ``coefficients''
$(a(q),\ell(q),Q(q),N(q,\bu))$.

\bigskip\noindent\textbf{Assumption A.}
\begin{compactenum}[(i)]
\label{assumHFF}
\item  We assume that  the function $H(q,p)$ is such that  $-H(\cdot, D)$ is closable and the closure is the
generator of a strongly continuous semigroup on $C_\infty(\real^d)$.

\item
 We assume also that the set $C_c^\infty(\real^d)$ of test
functions is an operator core for this generator.
\end{compactenum}

\bigskip
\begin{remark}
Conditions on the function $H(q,p)$ to fulfill Assumption
A (i) can be found, for example,  in Vol. 2 of
\cite{NJ} (Thms. 2.6.4, 2.6.9, 2.7.9, 2.7.16, 2.7.19, 2.8.1) or
in \cite{JSch}. For all these constructions $C^\infty_c(\real^d)$ is always an operator core. Note that
 Assumption  A (ii) holds also for example for
generators of L\'{e}vy processes, see \cite{Sato}
(Theo.~31.5).
\end{remark}
\bigskip

Let $F(t)$ be a $\Psi$DO with the symbol  $e^{-tH(q,p)}$, i.e.
for each $\ffi\in C_c^\infty(\real^d)$
\begin{equation}\label{F(t)}
F(t)\ffi(q)=(2\pi)^{-d/2}\intl_{\real^d}e^{ip\cdot q}e^{-tH(q,p)}\widehat{\ffi}(p)dp.
\end{equation}

\begin{lemma}
 For each
$\ffi\in C_c^\infty(\real^d)$ the function $F(t)\ffi$ belongs to
$C_\infty(\real^d)$.
\end{lemma}

\begin{proof}
The Fourier transform $\widehat \ffi$ of a
test function $\ffi\in C_c^\infty(\rn)$ is in the Schwartz space
$S(\rn)$ of rapidly decreasing functions. Since  $q\mapsto
e^{-tH(q,p)}$ is continuous (by assumption \eqref{c}) and bounded
($\Re H\ge 0$ due to properties of continuous negative definite
functions), Lebesgue's Dominated Convergence theorem shows that
$F(t)$ maps $C_c^\infty(\rn)$ into $C(\rn)$.

Let us prove, that $F(t)\ffi(q)\to 0$ when $|q|\to\infty$. Since
$H(q,\bu)$ is continuous negative definite for all $q\in\rn$ then
$e^{-tH(q,\bu)}$ is continuous positive definite for all $q\in\rn$
and all $t>0$ due to Theorem \ref{s3}. Then  the function
$$\bigl[p\mapsto
h_t(q,p):=e^{-tH(q,0)}-e^{-tH(q,p)}\bigr]
$$ is also continuous negative
definite for all $q\in\rn$ by Corollary \ref{cor}. Hence,
$h_t(q,\bu)$ satisfies  a L\'evy-Khintchine representation
\begin{equation}\label{g1160}
    h_t(q,p)
    = a_t(q) +i\ell_t(q)\cdot p + p\cdot Q_t(q)p
      + \intl_{y\neq 0}
        \left(1-e^{iy\cdot p} + \frac{iy\cdot p}{1+|y|^2} \right)\,N_t(q,dy),
\end{equation}
where $(a_t(q), \ell_t(q), Q_t(q),N_t(q,\bu))$ is for each $q\in\rn$
the L\'evy characteristics of $h_t(q,\bu)$. Again we can consider a
$\Psi$DO $h_t(q,D)$ with the symbol $h_t(q,p)$, i.e. for each
$\ffi\in C^\infty_c(\rn)$
\begin{equation}\label{g350}\begin{aligned}
    h_t(q,D)\ffi(q)&
    =(2\pi)^{-d/2}\int_{\real^d}e^{ip\cdot q}{h_t(q,p)}\widehat{\ffi}(p)dp\\
    &=
    -a_t(q)\ffi(q) + \ell_t(q)\cdot\nabla \ffi(q)
    + \sum_{j,k=1}^d Q_t^{jk}(q)\partial_j\partial_k u(q)\\
    &\qquad\mbox{} + \intl_{y\neq 0} \left( \ffi(q+y) - \ffi(q)
    - \frac{y\cdot\nabla \ffi(q)}{1+|y|^2}\right)\,N_t(q,dy)
\end{aligned}
\end{equation}

Note, that
\begin{equation*}
F(t)\ffi(q)=(2\pi)^{-d/2}e^{-tH(q,0)}\intl_{\real^d}e^{ip\cdot q}\widehat{\ffi}(p)dp\,-\,
(2\pi)^{-d/2}\intl_{\real^d}e^{ip\cdot q}{h_t(q,p)}\widehat{\ffi}(p)dp.
\end{equation*}
Since $\Re H\ge 0$ then $\sup\limits_{q\in\rn}\big|
e^{-tH(q,0)}\big|\le1$, and the first integral in the above formula tends to zero as $|q|\to\infty$
by the Riemann--Lebesgue Theorem. Thus, we only need to
show that
$$
\bigg[q\mapsto
(2\pi)^{-d/2}\intl_{\real^d}e^{ip\cdot q}{h_t(q,p)}\widehat{\ffi}(p)dp\bigg]\in
C_\infty(\rn).
$$

As $\ffi$ has compact support, there is some $R>0$ such that $\supp
\ffi\subset B_R(0)$. For all $|q|>2R$ formula (\ref{g350}) becomes
\begin{align*}
    |h_t(q,D)\ffi(q)|
    &= \left|\, \intl_{\,y\neq 0} \ffi(q+y)\,N_t(q,dy)\right|\\
    &= \left|\, \intl_{\,|y| > R} \ffi(q+y)\,N_t(q,dy)\right|\\
    &\leq 2\intl_{\,y\neq 0} \frac{|y/R|^2}{1+ |y/R|^2}\,N_t(q,dy)\cdot\nnorm \ffi_\infty.
\end{align*}
{\allowdisplaybreaks The last line follows from the elementary
inequality $\frac 12 \leq \frac{t^2}{1+t^2}$ for $|t|>1$ which
applies if $|y|>R$, and from $\ffi(q+y) = 0$ if $|q|>2R$ and
$|y|\leq R$. We can now use Lemma \ref{l1},   the L\'evy-Khintchine
representation of $h_t(q,\bu)$ and the estimate \eqref{g25} for a
continuous negative definite function $h_t(q,\frac{\bu}{R})$ to get
\begin{align*}
    |h_t(q,D)\ffi(q)|
    &\leq
        2\intl_{y\neq 0}
        \intl_\rn\left(1-\cos\frac{y\cdot\eta}{R}\right)\,g(\eta)\,d\eta\,N_t(q,dy)\cdot\nnorm \ffi_\infty\\
    &\le
       2\intl_\rn\Re h_t\left(q,\frac\eta R\right)\,g(\eta)\,d\eta\cdot\nnorm \ffi_\infty\\
    &\leq
       2\intl_\rn\left| h_t\left(q,\frac\eta R\right)\right|\,g(\eta)\,d\eta\cdot\nnorm \ffi_\infty\\
    &\leq
        2\sup_{|\xi|\leq 1/R} |h_t(q,\xi)|\int_\rn \left(1+|\eta|^2\right)\,g(\eta)\,d\eta\cdot\nnorm \ffi_\infty.
\end{align*}
}
Since $g(\eta)$ has absolute moments of any order, we see
$$
    |h_t(q,D)\ffi(q)| \leq c_g\, \sup_{q\in\rn} \sup_{|\xi|\leq 1/R} |h_t(q,\xi)|\cdot\nnorm \ffi_\infty
    \qquad\textrm{for all}\quad |q| > 2R.
$$
As  $h_t(q,0)=0$, the condition \eqref{b} tells us that
$\lim_{|q|\to\infty} h_t(q,D)\ffi(q) = 0$. Therefore, $h_t(q,\bu)\ffi\in
C_\infty(\rn)$.
\end{proof}

\bigskip

\begin{lemma}\label{contraction}
For each $t>0$  a mapping $F(t)$ can be extended to a contraction
$F(t): \cont_\infty(\real^d)\to\cont_\infty(\real^d)$.
\end{lemma}
\begin{proof}
 Let us freeze the coefficients (see e.g.
\cite{JP}). For each $t>0$ and each $q_0\in\real^d$ let us
consider a $\Psi$DO $F^{q_0}(t)$ with the symbol $e^{-tH(q_0,p)}$,
i.e. for any $\ffi\in C_c^\infty(\real^d)$ we have
$$F^{q_0}(t)\ffi(q)=(2\pi)^{-d/2}\intl_{\real^d}
e^{ip\cdot q}e^{-tH(q_0,p)}\widehat{\ffi}(p)dp.
$$
Then $F(t)\ffi(q)=F^q(t)\ffi(q)$ for any $\ffi\in C_c^\infty(\real^d)$
and any $q\in\real^d$. Since for each $q_0\in\real^d$  the function
$e^{-tH(q_0,\cdot)}$ is positive definite then there exists a
convolution semigroup $(\mu^{q_0}_t)_{t\ge0}$, such that
$\mathcal F^{-1}[\mu^{q_0}_t]=(2\pi)^{-d/2}e^{-tH(q_0,\cdot)}$ and
$F^{q_0}(t)\ffi(q)=\int_{\real^d} \ffi(q-y)\mu^{q_0}_t(dy)$. Hence,
for each $q_0\in\real^d$ a family $(F^{q_0}(t))_{t\ge 0}$ is a Feller
semigroup, and  for each $q,\, q_0\in\real^d$ we have
$$
\left| F^{q_0}(t)\ffi(q)  \right|=\left| \,\intl_{\,\real^d}
\ffi(q-y)\mu^{q_0}_t(dy)\right|\le \|\ffi\|_\infty.
$$
Then
$\|F(t)\ffi\|_\infty=\sup_{q\in\real^d}|F(t)\ffi(q)|=\sup_{q\in\rn}|F^q(t)\ffi(q)|\le
\|\ffi\|_\infty$ for any $\ffi\in S(\real^d)$. Hence, the family $F(t)$ can be
extended to a contraction from  $\cont_\infty(\real^d)$ into itself by
the B.L.T. Theorem (see \cite{ReedSimon}, p.9).
\end{proof}

\begin{theorem}\label{HFF} Let the function  $H:\rn\times\rn\to\comp$ be  measurable and locally bounded in both
variables $(q,p)$. Assume that    $H(q,\bu)$ is continuous and
negative definite  for all $q\in\rn$ and that  conditions
\eqref{a}, \eqref{b} and \eqref{c} hold. Under Assumption
A the family $(F(t))_{t\ge0}$ is Chernoff equivalent
to a strongly continuous semigroup $(T_t)_{t\ge0}$, generated by the
closure of the $\Psi$DO $-{H}(\cdot,D)$ with  the symbol $-H(q,p)$, and
the Hamiltonian Feynman formula $T_t=\lim_{n\to\infty}\left[
F(\frac{t}{n})\right]^n$ is valid in $L(C_\infty(\real^d))$  locally
uniformly with respect to $t \ge 0$.
\end{theorem}

\begin{proof}
 By Lemma \ref{contraction} each $F(t)$ is a
contraction operator on $\cont_\infty(\real^d)$, thus we only need
to prove that for all $\ffi\in\cont_\infty(\real^d)$ we have
$\lim_{t\to0}\|F(t)\ffi-\ffi\|_\infty=0$ and  $F'(0)=-{H}(\cdot, D)$
on a core of $-{H}(\cdot, D)$.

 For any $\ffi\in C_c^\infty(\real^d)$ we have, due to the estimate \eqref{a},
 \begin{align*}
\lim_{t\to0}\|F(t)\ffi-\ffi\|_\infty
&=\lim_{t\to0}\sup_{q\in\real^d}\left| (2\pi)^{-d/2}\intl_{\,\real^d} e^{ip\cdot q}\widehat{\ffi}(p)\left[
e^{-tH(q,p)}-1\right]dp\right|\\
&\le\lim_{t\to0} (2\pi)^{-d/2}\intl_{\real^d} |\widehat{\ffi}(p)|
\sup_{q\in\real^d}\left\{\left|\frac{e^{-tH(q,p)}-1}{-tH(q,p)}\right|\left|
tH(q,p)\right|\right\}dp\\
&\le\lim_{t\to0} (2\pi)^{-d/2}\intl_{\real^d}
t\kappa(1+|p|^2)|\widehat{\ffi}(p)| dp\\
&=0,
\end{align*}
since  $\widehat{\ffi}\in S(\real^d)$. Hence,
$\lim_{t\to0}\|F(t)\ffi-\ffi\|_\infty=0$ for all $\ffi\in
C_c^\infty(\real^d)$. As $\|F(t)\|\le 1$, then the last equality is
true for all $\ffi\in\cont_\infty(\real^d)$ by a 3-epsilon argument.

In a similar way, for any $\ffi\in C_c^\infty(\real^d)$ we have
\begin{align*}
\lim_{t\to0}\bigg\|\frac{F(t)\ffi-\ffi}{t}&+H(\cdot,D)\ffi\bigg\|_\infty\\
&=\lim_{t\to0}\sup_{q\in\real^d}\left| (2\pi)^{-d/2}\intl_{\real^d} e^{ip\cdot q}\widehat{\ffi}(p)\left[
\frac{e^{-tH(q,p)}-1}{t}+(H(q,p))\right]dp\right|\\
&\le\lim_{t\to0} (2\pi)^{-d/2}\int_{\real^d}
|\widehat{\ffi}(p)|\frac{t\kappa^2(1+|p|^2)^2}{2} dp\\
&=0.
 \end{align*}
Thus, all assumptions of Chernoff's theorem are fulfilled, and the
family $F(t)$ is Chernoff equivalent to the semigroup $T_t$
generated by $-H(\cdot,D)$.
\end{proof}

\begin{remark}\label{properSense}
(i)
Let us assume additionally that $H:\rn\times\rn\to\comp$ satisfies the following condition:
\begin{equation*}
\exists\, C>0\quad \,\mbox{such that }\,
\big\|\partial^\alpha_q\partial^\beta_p
e^{tH}\big\|_{L^\infty(\rn\times\rn)}\le C,
\end{equation*}
where $ \alpha,\,\beta\in \mathds{N}_0^d$,  $\, \alpha=0\,\mbox{ or
} \,1$, $  \beta=0\,\mbox{ or } \,1$,
$\,\partial^\alpha_q\partial^\beta_p$ are derivatives in  the distributional
sense. Note, that this condition is fulfilled, e.g. if
$H:\,|H(q,p)|\ge c|p|^r$ for $|p|\gg1$, some $c>0$ and some
$r\in(0,2)$. Then by Theorem 2 of Ref. \cite{Hwang87} we have
 $F(t): L_2(\real^d)\to L_2(\real^d)$.
 In this case the Hamiltonian Feynman formula obtained in
Theorem \ref{HFF} has the following form:
\begin{align}
\label{hff}&(T_t\ffi)(q_0)\\
\notag&=\lim\limits_{n\rightarrow\infty}\frac 1{(2\pi)^{dn}}\int\limits_{(\real^d)^{2n}}
e^{i\sum\limits_{k=1}^n p_k\cdot(q_{k-1}-q_k)}e^{-\frac{t}{n}\sum\limits_{k=1}^n
H(q_{k-1},p_k)}\ffi(q_n)dq_1dp_1\cdots dq_ndp_n,
\end{align}
where the equality holds  in $L_2$-sense (i.e. the integrals in the right hand side must be considered in a regularized sense).  We refer to \cite{Hwang87}
for further  conditions on $H(q,p)$ ensuring   $F(t): L_2(\real^d)\to L_2(\real^d)$.

\bigskip
\noindent (ii)  If   the function $H$ satisfies sufficient
conditions  for $F(t)\ffi$ to be in $ S(\real^d)$ for each $\ffi\in
S(\real^d)$ then for any $\ffi\in S(\real^d)$ the equality in the
Hamiltonian Feynman formula \eqref{hff} holds  in each point
$q_0\in\real^d$. Such conditions can be found in the following
lemma.
\end{remark}

\begin{lemma}
 Let $H: \real^d\times\real^d\to\comp$ be a continuous function
such that for each $q\in\real^d$ a mapping $p\mapsto  H(q,p)$ is
negative definite and
$H(\cdot,\cdot)\in\cont^\infty(\real^d\times\real^d)$.  Assume that for each
$p\in\real^d$ and for each $\alpha,\,\beta\in \mathds{N}^d_0$,
  the following estimates
hold:
\begin{equation}\label{coefficients2}
\sup_{q\in\real^d}|\partial^\alpha_p\partial^\beta_q H(q,p)|\le
f_{\alpha,\beta}(p),
\end{equation}
where all functions $f_{\alpha,\beta}$ are continuous on $\real^d$ and
have at most  polynomial growth at infinity. Then $F(t)\ffi\in
S(\real^d)$ for each $\ffi\in S(\real^d)$.
\end{lemma}
\begin{proof}
 By Lemma \ref{contraction} we have
$F(t)\ffi\in C_\infty(\real^d)$. Let us show that for all
$\alpha,\,\beta\in \mathds{N}^d_0$  the norm
$$\|F(t)\ffi\|_{\alpha,\beta}=\sup_{q\in\real^d}\big| q^\alpha\partial^\beta_q[F(t)\ffi](q)
\underline{}\big|$$ is finite. Note that for any $\beta\in
\mathds{N}^d_0$ the function $\partial^\beta_q e^{-tH(q,p)+ip\cdot q}$ is
continuous. By  \eqref{coefficients2} it is also majorized
(uniformly for all $q$) by some continuous function of $p$ which has
at most polynomial growth at infinity. Hence, by  Lebesgue's
dominated convergence theorem, we have
\begin{align*}
q^\alpha\partial^\beta_q[F(t)\ffi](q)&=(2\pi)^{-d/2}\intl_{\real^d}q^\alpha\partial^\beta_q
e^{-tH(q,p)+ip\cdot q}\widehat{\ffi}(p)dp\\
&=(2\pi)^{-d/2}\sum_{0\le\gamma\le\beta}\,\intl_{\real^d}q^\alpha
\partial^\gamma_q(e^{ip\cdot q})\partial^{\beta-\gamma}_q(e^{-tH(q,p)})\widehat{\ffi}(p)dp.
\end{align*}
Since $\partial^\gamma_q(e^{ip\cdot q})=e^{ip\cdot q}R_\gamma(p)$, where $R_\gamma$
is a polynomial of $p$, we can use integration by parts and get
\begin{align*}
q^\alpha\partial^\beta_q[F(t)\ffi](q)&=(2\pi)^{-d/2}\sum_{0\le\gamma\le\beta}\,\intl_{\real^d}q^\alpha
e^{ip\cdot q}\big[R_\gamma(p)\partial^{\beta-\gamma}_q(e^{-tH(q,p)})\widehat{\ffi}(p)\big]dp\\
&=
(2\pi)^{-d/2}\sum_{0\le\gamma\le\beta}i^{|\alpha|}\intl_{\real^d}\partial^\alpha_p
e^{ip\cdot q}\big[R_\gamma(p)\partial^{\beta-\gamma}_q(e^{-tH(q,p)})\widehat{\ffi}(p)\big]dp\\
&=
(2\pi)^{-d/2}\sum_{0\le\gamma\le\beta}(-i)^{|\alpha|}\intl_{\real^d}
e^{ip\cdot q}\partial^\alpha_p\big[R_\gamma(p)\partial^{\beta-\gamma}_q(e^{-tH(q,p)})\widehat{\ffi}(p)\big]dp.
\end{align*}
Since $\partial^\alpha_p\big[R_\gamma(p)\partial^{\beta-\gamma}_q(e^{-tH(q,p)})\widehat{\ffi}(p)\big]$
is bounded by an $L_1$-function which is independent of $q$, we can use \eqref{coefficients2} to see that the expression in the last line is finite. Hence,
the norm $\|F(t)\ffi\|_{\alpha,\beta} $ is finite.
\end{proof}

\begin{remark}\label{approx} If in Assumption A (i) we require the
existence of not just a strongly continuous but a Feller semigroup,
we obtain a Hamiltonian Feynman formula for the corresponding Feller
process $\xi_t$. Besides, for each fixed $n$ the operator
$\big[F(t/n)\big]^n$ in the Hamiltonian  Feynman formula corresponds
to the approximation of the process $\xi_t$ by a Markov chain
$\big\{ Y^{t/n}(k) \big\}_{k=0}^n$ with L\'{e}vy increments. This
Markov chain is obtained by  splitting  the time interval $[0,t]$
onto $n$ equal steps and  ``freezing''  the coefficient $q$  in the
transition probabilities of $\xi_t$ at each step (cf.
\cite{schilling-boettcher}). Moreover, the transition kernels
$\mu_{q,t/n}$ of this Markov chain correspond to the transition
operator $W_{t/n}$,  $W_{t/n}\ffi(q)=\int_{\real^d}\ffi(y)\mu_{q,
t/n}(dy)$. Hence,  $\big[F(t/n)\big]^n=\big[W_{t/n}\big]^n$.
This allows us to transform the obtained Hamiltonian Feynman formula
for the Feller semigroup $T_t$ associated with the process $\xi_t$
into a Lagrangian Feynman formula
$T_t\ffi(q)=\lim_{n\to\infty}\big[W_{t/n}\big]^n\ffi(q)$.
\end{remark}
\begin{example}
Let us consider the symbol $H_1(q,p)=a(q)|p|^\alpha$, where
$\alpha\in(0,2]$ and $a(\cdot)\in C^\infty(\rn)$ is  a strictly
positive and bounded function. Then $-H_1(\bu,D)$ generates a Feller
semigroup $(T^1_t)_{t\ge 0}$ (see \cite{schilling-schnurr}).
If $\alpha=2$ this semigroup corresponds to the  process of
diffusion with  variable diffusion coefficient. All conditions of
the Theorem \ref{HFF} are fulfilled
 and  by the Hamiltonian Feynman
formula \eqref{hff}   for any $\ffi\in C_c^\infty(\real^d)$ and any
$q_0\in\real^d$ we have:
\begin{equation*}\begin{aligned}
&(T_t^1\ffi)(q_0)\\
&=\lim_{n\rightarrow\infty}\frac{1}{(2\pi)^{dn}}\int\limits_{\real^{2dn}}
e^{i\sum\limits_{k=1}^n p_k\cdot(q_{k-1}-q_k)}e^{-\frac{t}{n}\sum\limits_{k=1}^n
a(q_{k-1})|p_k|^\alpha}\ffi(q_n)dq_1dp_1\cdots dq_ndp_n.
\end{aligned}\end{equation*}
\end{example}
\begin{example}
Let us consider the symbol
$H_2(q,p)=\sqrt{|p|^\alpha+m^2(q)}-m(q)$, where $m(\cdot)\in
C^\infty(\real^d)$ is a strictly positive  and bounded function on $\real^d$,
$\alpha\in(0,2]$. If additionally    the function
$m(\cdot)$ is such that the Assumption A holds (e.g.
if $m\equiv\text{const}$), then the following Hamiltonian
Feynman formula is valid for the corresponding semigroup
$(T^2_t)_{t\ge 0}$:
\begin{multline*}
(T_t^2\ffi)(q_0)
=\lim_{n\rightarrow\infty}\frac{1}{(2\pi)^{dn}}\int\limits_{\real^{2dn}}
e^{i\sum\limits_{k=1}^n p_k\cdot(q_{k-1}-q_k)}
e^{-\frac{t}{n}\sum\limits_{k=1}^n
\sqrt{|p_k|^\alpha+m^2(q_{k-1})}-m(q_{k-1})}\times\\
\times\ffi(q_n)dq_1dp_1\cdots dq_ndp_n,
\end{multline*}
 where
 $\ffi\in C_c^\infty(\real^d)$ and  $q_0\in\real^d$.  In the case $\alpha=2$ the operator $H_2(\bu,D)$
 can be considered as a Hamiltonian of a free relativistic (quasi-)particle
 with  variable mass (cf. \cite{GS}, \cite{IchTam86}).
\end{example}

\section{\bfseries Lagrangian Feynman formula for  multiplicative perturbations of Feller semigroups}
Let $H(\cdot,D)$ be a $\Psi$DO with the symbol $H(q,p)$.
Assume that $-H(\cdot,D)$ generates a Feller semigroup $(T_t)_{t\ge0}$,
$$
T_t\ffi(q)=\intl_{\real^d}\ffi(y)p_t(q,dy),
$$
where  $p_t(q,dy)=\mathds{P}^q(X_t\in dy)$ is the transition probability of the underlying Feller process $X_t$,  $\ffi\in C_\infty(\real^d)$.
Let $a(\cdot):\real^d\to [c,1/c]$, $c>0$, be a continuous function. Then (due to \cite{dorroh}) a $\Psi$DO $-\widetilde{H}(\cdot,D)$ with the symbol $-\widetilde{H}(q,p)=-a(q)H(q,p)$  generates a strongly continuous  semigroup $(\widetilde{T}_t)_{t\ge0}$ on the space $C_\infty(\real^d)$.

Consider now a family $(\widetilde{F}(t))_{t\ge0}$ of operators on $C_\infty(\real^d)$ defined by the formula:
\begin{equation}\label{F(t)-LFF}
\widetilde{F}(t)\ffi(q)=(T_{a(q)t}\ffi)(q)\equiv\intl_{\real^d}\ffi(y)p_{a(q)t}(q,dy)
\end{equation}
\begin{theorem}\label{LFF!!!}
The family $(\widetilde{F}(t))_{t\ge0}$ given by the formula \eqref{F(t)-LFF} is Chernoff equivalent to the semigroup $(\widetilde{T}_t)_{t\ge0}$ generated by  a $\Psi$DO $-\widetilde{H}(\cdot,D)$ with the symbol $-\widetilde{H}(q,p)=-a(q)H(q,p)$ and, hence, the   Lagrangian  Feynman
formula
\begin{equation} \label{lFF}\begin{aligned}
\widetilde{T}_t\ffi(q_0)=
\lim_{n\to\infty}\intl_{\real^d}\cdots \intl_{\real^d}\ffi(q_n)
p_{a(q_0)t/n}&(q_0,dq_1)p_{a(q_1)t/n}(q_1,dq_2)\cdots\\
&\cdots p_{a(q_{n-1})t/n}(q_{n-1},dq_n)
\end{aligned}\end{equation}
is valid in $L(C_\infty(\real^d))$  locally
uniformly with respect to $t \ge 0$.
\end{theorem}

\begin{remark}
The transformation of a Feller process under which a symbol
$H(q,p)$ transfers into  a symbol $a(q)H(q,p)$ may be understood as
a position-dependent time re-scaling of the process: $t\rightsquigarrow
a(q)t$. Indeed, for each random variable $X_{a(q)t}$ we have:
$$
\frac{\mathds{E}^q\big[e^{i(X_{a(q)t}-q)\cdot p}  \big]-1}{t}
=a(q)\frac{\mathds{E}^q\big[e^{i(X_{a(q)t}-q)\cdot p}  \big]-1}{a(q)t}
\longrightarrow -a(q)H(q,p),\quad t\to 0.
$$
 Note, that $p(t,x,dy)=p_{a(x)t}(x,dy)$ is NOT the transition
function of the re-scaled process.

If the explicit form of the transition density of the original
Feller  process is known then the Lagrangian Feynman formula
\eqref{lFF}  for the time re-scaled process  contains only explicit---usually
elementary---functions.
\end{remark}

\begin{proof}
First, let us prove that the family $(\widetilde{F}(t))_{t\ge0}$ acts in the space $C_\infty(\mathbb{R}^d)$. For any fixed $\ffi\in C_\infty(\mathbb{R}^d)$  we have
\begin{align*}
\lim\limits_{q\to q_0}|&\widetilde{F}(t)\ffi(q)-\widetilde{F}(t)\ffi(q_0)|\\
&=\lim\limits_{q\to q_0}|(T_{a(q)t}\ffi)(q)-(T_{a(q_0)t}\ffi)(q_0)|\\
&\le\lim\limits_{q\to q_0}\bigg(|(T_{a(q)t}\ffi)(q)-(T_{a(q_0)t}\ffi)(q)|+|(T_{a(q_0)t}\ffi)(q)-(T_{a(q_0)t}\ffi)(q_0)|\bigg)\\
&
\le \lim\limits_{q\to q_0}\|[T_{a(q)t}-T_{a(q_0)t}]\ffi\|_\infty+\lim\limits_{q\to q_0}|(T_{a(q_0)t}\ffi)(q)-(T_{a(q_0)t}\ffi)(q_0)|\\
&=0.
\end{align*}
Therefore, the function $\widetilde{F}(t)\ffi$ is continuous. Since
\begin{align*}
\lim\limits_{|q|\to\infty}|\widetilde{F}(t)\ffi(q)|
&=\lim\limits_{|q|\to\infty}|(T_{a(q)t}\ffi)(q)|\\
&\le\lim\limits_{|q|\to\infty} \sup\limits_{q_0\in \mathbb{R}^d}|T_{a(q_0)t}\ffi(q)|\\
&=\lim\limits_{|q|\to\infty} \sup\limits_{s\in[ct,t/c]}|T_{s}\ffi(q)|,
\end{align*}
we get that $\widetilde{F}(t):C_\infty(\mathbb{R}^d)\to C_\infty(\mathbb{R}^d)$.

Since  the semigroup   $(T_t)_{t\ge0}$ acts in $C_\infty(\mathbb{R}^d)$,  then for any  $t\ge0$ and $\eps>0$  there exists  $R_{\eps,t}>0$ such that for any  $q\in \mathbb{R}^d \,:\,|q|>R_{\eps,t}$ the inequality  $|T_t\ffi(q)|<\frac12\eps$ holds. Due to the strong continuity of $(T_t)_{t\ge0}$ there exists  $\delta_\eps>0$  such that for all  $\tau,\tau'\in[ct,t/c]$ with  $|\tau-\tau'|<\delta_\eps$ the inequality  $\|T_\tau\ffi-T_{\tau'}\ffi\|_\infty<\frac12\eps$ holds.
Let us fix  $\eps>0$. Consider  a partition  $\tau_0=ct<\tau_1<\ldots<\tau_N=t/c$ of a segment  $[ct,t/c]$ such that  $\max\limits_{1\le k\le N}|\tau_k-\tau_{k-1}|<\delta_\eps$. Then for any  $\tau\in[ct,t/c]$ ther exists  $\tau_k$ with  $|\tau-\tau_k|<\delta_\eps$. Let now $R_\eps=\max\limits_{0\le k\le N}R_{\eps,\tau_k}$. Then for any  $q\in \mathbb{R}^d\,:\, |q|>R_\eps$ and any  $\tau\in[ct,t/c]$  we have
\begin{align*}
|T_\tau\ffi(q)|&\le|T_\tau\ffi(q)-T_{\tau_k}\ffi(q)|+ |T_{\tau_k}\ffi(q)|\\
&
\le \|T_{\tau}\ffi-T_{\tau_k}\ffi\|_\infty+ |T_{\tau_k}\ffi(q)|\\
&
\le\frac12\eps+\frac12\eps=\eps.
\end{align*}
Therefore, $\lim\limits_{|q|\to\infty} \sup\limits_{\tau\in[c_1t,c_2t]}|T_\tau\ffi(q)|=0$ and, hence,  the function   $\widetilde{F}(t)\ffi$ vanishes at infinity.

Further we use a freezing-in technique.  For each $q_0\in\real^d$ consider a family of operators $(F^{q_0}(t))_{t\ge 0}$  on $C_\infty(\real^d)$ such that
\begin{equation*}
F^{q_0}(t)\ffi(q)=\intl_{\real^d}\ffi(y)p_{a(q_0)t}(q,dy)\equiv (T_{a(q_0)t}\ffi)(q).
\end{equation*}
Then $\widetilde{F}(t)\ffi(q)=F^q(t)\ffi(q)$ for all $\ffi\in C_\infty(\real^d)$, $q\in\real^d$. Moreover,  $\widetilde{F}(0)=T_0=\id$ and, as $(T_t)_{t\ge 0}$ is a contraction semigroup, we have
\begin{align*}
\|\widetilde{F}(t)\ffi\|_\infty&\le \sup_{q_0\in\real^d}\sup_{q\in\real^d}\big|F^{q_0}(t)\ffi(q)\big|\\
&=\sup_{q_0\in\real^d}\|T_{a(q_0)t}\ffi\|_\infty\\
&\le \|\ffi\|_\infty.
\end{align*}
The family $(\widetilde{F}(t))_{t\ge0}$ is strongly continuous since
\begin{align*}
\lim_{t\to0}\|\widetilde{F}(t)\ffi-\ffi\|_\infty&=\lim_{t\to0}\sup_{q\in\real^d}|F^q(t)\ffi(q)-\ffi(q)|\\
&\le \lim_{t\to0} \sup_{q_0\in\real^d}\sup_{q\in\real^d}|F^{q_0}(t)\ffi(q)-\ffi(q)|\\
&=\lim_{t\to0} \sup_{q_0\in\real^d}\|T_{a(q_0)t}\ffi-\ffi\|_\infty\\
&\leq \lim_{t\to0} \sup_{a\in[c,1/c]}\|T_{at}\ffi-\ffi\|_\infty\\
&=0.
\end{align*}
And for all $\ffi\in D(-H(\cdot,D))$ we have
\begin{align*}
\bigg\|\frac{\widetilde{F}(t)\ffi-\ffi}{t}&+\widetilde{H}(\cdot,D)\ffi\bigg\|_\infty\\
&=
\sup_{q\in\real^d}\bigg|\frac{{F^q}(t)\ffi(q)-\ffi(q)}{t}+a(q){H}(q,D)\ffi(q)\bigg|\\
&\le\sup_{q_0\in\real^d} \sup_{q\in\real^d}\bigg|\frac{F^{q_0}(t)\ffi(q)-\ffi(q)}{t}+a(q_0){H}(q,D)\ffi(q)\bigg|\\
&=\sup_{q_0\in\real^d} \sup_{q\in\real^d}\bigg| \frac{-1}{a(q_0)t}\intl_0^{a(q_0)t}a(q_0) H(q,D)(T_\tau\ffi(q)-\ffi(q))d\tau  \bigg|\\
&\le  \frac1t \int_0^{t/c}\| H(q,D)(T_\tau\ffi(q)-\ffi(q))\|_\infty d\tau\\
&\longrightarrow 0,\quad t\to0.
\end{align*}
Therefore, all assumptions of Chernoff's theorem are fulfilled and, hence, the family $(\widetilde{F}(t))_{t\ge0}$ is Chernoff equivalent to the semigroup $(\widetilde{T}_t)_{t\ge0}$ generated by a $\Psi$DO $-\widetilde{H}(\cdot,D)=-a(\cdot)H(\cdot,D)$.
\end{proof}

\begin{remark}
One can show that (the appropriate modification of) Theorem \ref{LFF!!!} remains true  for multiplicative perturbations of not necessary  Feller but just  strongly continuous semigroups on $C_\infty(\real^d)$.
\end{remark}
\begin{remark}
If the symbol $-H(q,p)$ of a $\Psi$DO $-H(\cdot,D)$ generating a Feller semigroup $(T_t)_{t\ge0}$ satisfies the assumptions of Theorem \ref{HFF}, then the symbol $-\widetilde{H}(q,p)=-a(q)H(q,p)$ (with continuous $a(\cdot):\real^d\to[c,1/c]$, $c>0$) also satisfies the assumptions of this Theorem. Hence, the Hamiltonian Feynman formula obtained in Theorem \ref{HFF} remains valid for the perturbed semigroup $(\widetilde{T}_t)_{t\ge0}$ as well.
\end{remark}
\begin{example}[diffusion with variable diffusion
coefficient]\label{diffusion} Let $\psi(p)=\frac12|p|^2$ be the
characteristic exponent of a  Brownian motion in
$\real^d$. The generator of Brownian motion is
$-\psi(D)=\frac12\Delta$. The transition density is given by
Gaussian density $$p_t^{BM}(x)=(2\pi
t)^{-d/2}\exp\bigg\{-\frac{|x|^2}{2t}\bigg\}.$$ Consider the
semigroup $(\widetilde{T}_t)_{t\ge0}$, generated by a $\Psi$DO
$-{\widetilde{H}}(\cdot,D)$ with the symbol
$-\widetilde{H}(q,p)=-\frac12a(q)|p|^2$, where $a(\cdot)$ is as
before. Then by Theorem \ref{LFF!!!} for each $\ffi\in C_\infty(\real^d)$
we have (cf. \cite{BGS1}, \cite{BGS2}):
\begin{multline*} \widetilde{T}_t\ffi(q_0)=
\lim_{n\to\infty}\intl_{\real^d}\cdots \intl_{\real^d}(2\pi
a(q_0)t/n)^{-d/2}\exp\bigg\{-\frac{|q_0-q_1|^2}{2a(q_0)t/n}\bigg\}\cdots \\\cdots (2\pi
a(q_{n-1})t/n)^{-d/2}\exp\bigg\{-\frac{|q_{n-1}-q_n|^2}{2a(q_{n-1})t/n}\bigg\}
\ffi(q_n)dq_1\cdots dq_n
\end{multline*}
\end{example}
\begin{example}[Cauchy type process with variable coefficient] Let
$\psi(p)=|p|$ be the characteristic exponent of the Cauchy process
in $\real^d$. The generator of the Cauchy process is
$-\psi(D)=-\sqrt{-\Delta}$. The transition density is given by the
formula
$$p_t(x)=\Gamma\bigg(\frac{d}{2}+\frac12\bigg)\frac{t}{[\pi|x|^2+t^2]^{(d+1)/2}},$$
where $\Gamma(\cdot)$ is  Euler's Gamma function.

Consider the semigroup $(\widetilde{T}_t)_{t\ge0}$, generated by a
$\Psi$DO $-{\widetilde{H}}(\cdot,D)$ with the symbol
$-\widetilde{H}(q,p)=-a(q)|p|$, where $a(\cdot)$ is as before. Then
by Theorem \ref{LFF!!!} for each $\ffi\in C_\infty(\real^d)$ we have (cf.
\cite{BShS}):
\begin{multline*}
\widetilde{T}_t\ffi(q_0)=
\lim_{n\to\infty}\intl_{\real^d}\cdots \intl_{\real^d}
\bigg[\Gamma\bigg(\frac{d}{2}+\frac12\bigg)\bigg]^n
\frac{a(q_0)t/n}{[(a(q_0)t/n)^2+(\pi|q_0-q_1|)^2]^{(d+1)/2}}\cdots\\
\cdots\frac{a(q_{n-1})t/n}{[(a(q_{n-1})t/n)^2+(\pi|q_{n-1}-q_n|)^2]^{(d+1)/2}}\ffi(q_n)dq_1\cdots dq_n.
\end{multline*}
\end{example}

\section{Feynman formulae for additive perturbations of semigroups}

\begin{theorem}\label{TH5.1}
Let $X$ be a Banach space with a norm $\|\cdot\|_X$. Let
$(T_k(t))_{t\ge0}$, $k=1,\ldots,m,$   be  strongly continuous semigroups on $X$ with
generators $(A_k,D(A_k))$ respectively.
Assume that $A=A_1+\cdots +A_m$ with   domain $D(A)=\cap_{k=1}^m D(A_k)$ is closable and that the closure is the generator of a strongly continuous
semigroup $(T(t))_{t\ge0}$  on $X$.
Let  $(F_k(t))_{t\ge0}$, $k=1,\ldots ,m,$ be families of operators in $X$  which are
Chernoff equivalent to the semigroups $(T_k(t))_{t\ge0}$ respectively, i.e. for each $k\in\{1,\ldots ,m \}$ we have
$F_k(0)=\id$,  $\|F_k(t)\|\le e^{a_kt}$ for some $a_k>0$ and there is a set
$D_k\subset D(A_k)$, which is a core for $A_k$, such that  $\lim_{t\to0}\big\|
\frac{F_k(t)\ffi-\ffi}{t}-A_k\ffi\big\|_X=0$ for each $\ffi\in D_k$.  Assume that there exists  a set $D\subset \cap_{k=1}^m D_k$ which is a core for $A$.
Then the family  $(F(t))_{t\ge0}$, where $F(t)=F_1(t)\circ\cdots \circ F_m(t)$
is Chernoff equivalent to the semigroup $(T(t))_{t\ge0}$ and, hence, the Feynman formula
$$
T_t=\lim_{n\to\infty} \big[F(t/n)   \big]^n
$$
is valid in the strong operator topology locally uniformly with respect to $t\ge 0$.
\end{theorem}

\begin{proof}
Obviously, the family $(F(t))_{t\ge0}$ is strongly continuous, $F(0)=\id$ and
$$\|F(t)\|\le \|F_1(t)\|\cdot\ldots \cdot\|F_m(t)\|\le e^{(a_1+\cdots +a_m)t}.$$ Let  $D\subset \cap_{k=1}^m D_k$  be a core for $A$. Then
for each $\ffi\in D$ we have
\begin{align*}
\lim_{t\to0}&\bigg\|
\frac{F(t)\ffi-\ffi}{t}-A\ffi\bigg\|_X\\
&=\lim_{t\to0}\bigg\|
\frac{F_1(t)\circ\cdots \circ F_m(t)\ffi-\ffi}{t}-A_1\ffi-\cdots -A_m\ffi\bigg\|_X\\
&=
\lim_{t\to0}\bigg\|
F_1(t)\circ\cdots \circ F_{m-1}(t)\circ\frac{F_m(t)\ffi-\ffi}{t}- A_m\ffi\\
&\phantom{= \lim_{t\to0}MM}+\frac{F_1(t)\circ\cdots \circ F_{m-1}(t)\ffi-\ffi}{t}- A_1\ffi-\cdots -A_{m-1}\ffi   \bigg\|_X\\
& \le
\lim_{t\to0}\bigg\|
\frac{F_1(t)\circ\cdots \circ F_{m-1}(t)\ffi-\ffi}{t}-A_1\ffi-\cdots -A_{m-1}\ffi\bigg\|_X\\
& \le \cdots \le
\lim_{t\to0}\bigg\|
\frac{F_1(t)\ffi-\ffi}{t}-A_1\ffi\bigg\|_X\\
&=0.
\qedhere
\end{align*}
\end{proof}

Note, that if some of the $(T_k(t))_{t\ge0}$ are known explicitly and if $\|T_k(t)\|\le e^{a_kt}$ for some $a_k\ge0$ then  we can take $F_k(t)\equiv T_k(t)$ in the corresponding Feynman formulae.

\begin{example}[bounded Schr\"{o}dinger perturbations]
Let $X=C_\infty(\real^d)$ with the supremum norm, $(T_t)_{t\ge0}$  be a
strongly continuous  semigroup with a generator $(A,\, D(A))$, $(F(t))_{t\ge0}$ be
Chernoff equivalent to $(T_t)_{t\ge0}$. Let $V(\cdot):\real^d\to\real$ be a
bounded continuous function. Then an operator $A+V$, such that
$D(A+V)=D(A)$ and $(A+V)\ffi(q)=A\ffi(q)+V(q)\ffi(q)$ for all
$\ffi\in D(A+V)$, generates a strongly continuous semigroup
$(T^{A+V}_t)_{t\ge0}$ on $C_\infty(\real^d)$. By Theorem~\ref{TH5.1}
the Feynman formula
$T^{A+V}_t=\lim_{n\to\infty}\big[e^{\frac{t}{n}V}\circ F(t/n)\big]^n$
is valid. In particular, if  $\| T_t\|\le e^{at}$ for some $a\in[0,+\infty)$ and all $t\ge0$, then we have
$T^{A+V}_t=\lim_{n\to\infty}\big[e^{\frac{t}{n}V}\circ T_{t/n}\big]^n$.
In both formulae the operator $e^{tV}$ is an operator of
multiplication with the function $e^{tV}$.
\end{example}

\begin{example}[gradient perturbations]
Let again $X=C_\infty(\real^d)$ with the supremum norm.  Let
$b(\cdot):\real^d\to\real^d$ be a bounded continuous vector field and
$\nabla=(\frac{\partial}{\partial q_1},\ldots ,\frac{\partial}{\partial q_d} )$. Consider an
operator $b\nabla$, such that
${b\nabla}\ffi(q)=b(q)\nabla\ffi(q)$ for all $\ffi\in
D(b\nabla)$.
Consider a family $(S(t))_{t\ge0}$  of operators in
$C_\infty(\real^d)$ defined by the formula:
$$S(t)\ffi(q)=\ffi(q+tb(q)).
$$
Then  $(S(t))_{t\ge0}$  is Chernoff equivalent to the semigroup $e^{tb\nabla}$ generated by $b\nabla$.
Indeed, $S(0)=\id$, $\|S(t)\|=1$ and for all $\ffi\in D(b\nabla)$
\begin{align*}
\lim_{t\to0}\bigg\| \frac{S(t)\ffi-\ffi}{t}-b\nabla\ffi
\bigg\|_{\infty}&=\lim_{t\to0}\sup_{q\in\real^d}\bigg|
\frac{\ffi(q+tb(q))-\ffi(q)}{t}-b(q)\nabla\ffi(q) \bigg|\\
& \le
\lim_{t\to0}\,\, t \cdot\sup_{q\in\real^d,\,
s\in[0,t]}\big|b(q)\cdot \mathrm{Hess}\,\ffi(q+sb(q))b(q)\big|\\
&=0,
\end{align*}
where $\mathrm{Hess} \,\ffi$ is a Hessian of $\ffi$. Hence,
$\frac{d}{dt}S(t)\big|_{t=0}\ffi=b\nabla\ffi$ for each $\ffi\in D(b\nabla)$.

Let now  $(T_t)_{t\ge0}$ be a strongly continuous semigroup with a generator $(A, D(A))$ and
 a family  $(F(t))_{t\ge0}$ be Chernoff equivalent to  $(T_t)_{t\ge0}$. Consider an operator $A+b\nabla$
such that $(A+{b\nabla})\ffi(q)=A\ffi(q)+b(q)\nabla\ffi(q)$ for all $\ffi\in D(A+{b\nabla})=D(A)\cap D(b\nabla)$. Assume that
$A+{b\nabla}$ generates a strongly continuous semigroup $(T^{A+b\nabla}_t)_{t\ge0}$ on  $C_\infty(\real^d)$. Note, that the assumption holds, e.g. if the operator
$b\nabla$ is $A$-bounded, i.e. $D(A)\subset D(b\nabla)$ and for all $\ffi\in D(A)$, some $\lambda\in[0,1)$ and $\gamma\ge0$ the estimate
$$
\|b\nabla\ffi\|_X\le \lambda\|A\ffi\|_X+\gamma\|\ffi\|_X
$$
holds. In particular,  $b\nabla$ is $A$-bounded for   $A=-(-\Delta)^{\alpha/2}$, $\alpha\in(1,2]$.

By Theorem \ref{TH5.1} the family $(S(t)\circ F(t))_{t\ge0}$ is Chernoff equivalent to the semigroup $(T^{A+b\nabla}_t)_{t\ge0}$  and the Feynman formula
$T^{A+b\nabla}_t=\lim_{n\to\infty} \big[S(t/n)\circ F(t/n)  \big]^n$ is valid.

\end{example}

\begin{corollary}[Hamiltonian Feynman formula for gradient and
bounded Schr\"{o}\-din\-ger perturbations of Feller semigroups]
\label{perturb} Let $b(\cdot)$, $V(\cdot)$ be as in the above examples. Let a function $H(q,p)$ be as
in Theorem \ref{HFF},   ${H}(\cdot,D)$ be a $\Psi$DO with the symbol
$H(q,p)$ (see formula \eqref{H}) and $F(t)$ be given by
the formula \eqref{F(t)}.  We assume that the function $H(q,p)$ satisfies
sufficient conditions for  $-H(\cdot, D)$ and $A:=-H(\cdot,D)+b\nabla+V$ to be closable and the closures to generate strongly continuous semigroups on $C_\infty(\real^d)$.   Then by Theorems \ref{HFF},  \ref{TH5.1} and due to the above examples  the following Hamiltonian Feynman formula is valid for
the semigroup $(T_t^{A})_{t\ge 0}$, generated by
$A=-H(\cdot, D)+b\nabla+V$:
\begin{align*}
(T_t^{A}\ffi)(q_0)
&=\lim_{n\rightarrow\infty}\big[e^{\frac{t}{n}V}\circ
S(t/n)\circ {F}(t/n)
\big]^n\ffi(q_0)\\
&=\lim_{n\rightarrow\infty}({2\pi})^{-dn}\intl_{(\real^d)^{2n}}
e^{\frac{t}{n}\sum_{k=1}^n V(q_{k-1})} e^{i\sum_{k=1}^n
p_k\cdot(q_{k-1}-q_k+\frac{t}{n}b(q_{k-1}))}\times\\
&\quad\quad\quad\quad\quad\quad\mbox{}\times
 e^{-\frac{t}{n}\sum_{k=1}^n
H(q_{k-1}+\frac{t}{n}b(q_{k-1}),p_k)}\ffi(q_n)dq_1dp_1\cdots dq_ndp_n,
\end{align*}
where again the integrals in the formula must be understood in a
proper sense (see Remark \ref{properSense}).
\end{corollary}

\begin{corollary}[Lagrangian Feynman formula for a mixture of multiplicative and additive perturbations
of  Feller semigroups]
Let $b(\cdot)$, $V(\cdot)$ be as in the above examples, $\widetilde{H}(\cdot, D)$ be as
in Section 4, $\widetilde{F}(t)$---as in \eqref{F(t)-LFF}.  We assume that the function $\widetilde{H}(q,p)$ satisfies
sufficient conditions for  $-\widetilde{H}(\cdot, D)$ and $B:=-\widetilde{H}(\cdot,D)+b\nabla+V$ to be closable and the closures to generate strongly continuous semigroups on $C_\infty(\real^d)$.  Then by
Theorems \ref{LFF!!!}, \ref{TH5.1} and due to the above examples the following Lagrangian
Feynman formula is valid for the semigroup
$(\widetilde{T}^{B}_t)_{t\ge0}$, generated by
$B=-\widetilde{H}(\cdot,D)+b\nabla+V$:
\begin{align*}
&\widetilde{T}^{B}_t\ffi(q_0)\\
=&\lim_{n\rightarrow\infty}\big[e^{\frac{t}{n}V}\circ
S(\tfrac tn)\circ \widetilde{F}(\tfrac tn)
\big]^n\ffi(q_0)\\
=&
\lim_{n\to\infty}\intl_{\real^d}\cdots \intl_{\real^d}e^{\frac{t}{n}\sum_{k=1}^n
V(q_{k-1})}\ffi(q_n)
p_{a\big(q_0+b(q_0)\tfrac tn\big)\tfrac tn}\big(q_0+b(q_0)\tfrac tn,dq_1\big) \times\\
&\mbox{}\times p_{a\big(q_1+b(q_1)\tfrac tn\big)\tfrac tn}\big(q_1+b(q_1)\tfrac tn,dq_2\big)
\ldots  p_{a\big(q_{n-1}+b(q_{n-1})\tfrac tn\big)\tfrac tn}\big(q_{n-1}+b(q_{n-1})\tfrac tn,dq_n\big).
\end{align*}
\end{corollary}

\begin{example}[Lagrangian Feynman formula for perturbations of the heat semigroup]
Let again $b(\cdot)$, $V(\cdot)$ be as before. Consider $A=\frac12\Delta$; $A$ is  the generator of the heat
semigroup $(T_t)_{t\ge0}$: $T_t\ffi(q)=(2\pi
t)^{(-d/2)}\int_{\real^d}e^{-\frac{|q-y|^2}{2t}}\ffi(y)dy$. Then by
Theorem \ref{TH5.1}  the Lagrangian Feynman formula
is valid (cf. \cite{BGS1}, \cite{BGS2}) for the semigroup $(T^C_t)_{t\ge 0}$ generated by $C:=\frac12\Delta+b\nabla+V$:
\begin{equation*}\begin{aligned}
T^{C}_t\ffi(q_0)
= \lim_{n\to\infty}(2\pi t/n)^{(-dn/2)}
\intl\limits_{\real^d}\cdots \intl\limits_{\real^d}
e^{\frac{t}{n}\sum\limits_{k=1}^n V(q_{k-1})}
&e^{-\sum\limits_{k=1}^n\frac{|q_{k-1}+b(q_{k-1})t/n-q_k|^2}{2t/n}}\times\\
&\times\ffi(q_n)dq_1\cdots dq_n.
\end{aligned}\end{equation*}
Since $|x+b(x)t-y|^2=|x-y|^2+2tb(x)(x-y) +t^2|b(x)|^2$ then
\begin{multline*}
T^{b\nabla+V}_t\ffi(q_0)=
\lim_{n\to\infty}\intl_{\real^d}\cdots \intl_{\real^d}e^{\frac{t}{n}\sum_{k=1}^n
V(q_{k-1})} e^{-\sum_{k=1}^n b(q_{k-1})(q_{k-1}-q_k)}
\times\\\mbox{}\times
e^{-\frac{t}{2n}\sum_{k=1}^n |b(q_{k-1})|^2}p^{BM}_{t/n}(q_0-q_1)\cdots p^{BM}_{t/n}(q_{n-1}-q_n)
\ffi(q_n)dq_1\cdots dq_n,
\end{multline*}
where $p^{BM}_t(x)=(2\pi
t)^{(-d/2)}\exp\big\{-\frac{|x|^2}{2t}\big\}$ is the transition
density of  Brownian motion. Therefore, one can show, that the limit
in the right hand side of the last formula coincides with the
functional integral
\begin{equation}\label{Girsanov}
\mathds{E}^{q_0}\left[ e^{\int_0^t V(\xi_\tau)d\tau}e^{\int_0^t
b(\xi_\tau)d\xi_\tau}e^{-\frac12
\int_0^t|b(\xi_\tau)|^2d\tau}f(\xi_t)\right]
\end{equation}
with respect to  Wiener measure concentrated on the paths starting at
$q_0$. Hence, the machinery of Feynman formulae provides not only
another way to prove the Feynman--Kac formula \eqref{Girsanov} and
 Girsanov's formula but also to extend them for the case of
variable diffusion coefficients (cf. \cite{BGS2}).
\end{example}

\section*{Acknowledgements} Financial support by the Deutsche
Forschungsgemeinschaft through the project SCHI 419/7-1,
 by the Russian Foundation for Basic
Research through the project 10-01-00724-a, by the grant of the
President of Russian Federation through the projects MK-943.2010.1 and   MK-4255.2012.1, by the
Erasmus Mundus Action 2 Programme of the European Union
is gratefully acknowledged.

\end{document}